\documentclass[11pt]{article}
\RequirePackage[OT1]{fontenc}
\RequirePackage{amsthm,amsmath}
\RequirePackage{natbib}
\RequirePackage[colorlinks,citecolor=blue,urlcolor=blue]{hyperref}
\usepackage{graphicx}
\usepackage{float}
\usepackage{placeins}
\usepackage{booktabs}
\usepackage{fullpage}
\usepackage{econometrics}

\newtheorem{theorem}{Theorem}
\newtheorem{lemma}{Lemma}
\newtheorem{corollary}{Corollary}
\newtheorem{remark}{Remark}

\newcommand{\IQR}{\text{IQR}}
\newcommand{\Var}{\text{Var}}
\newcommand{\ASV}{\text{ASV}}
\newcommand{\IF}{\text{IF}}
\newcommand{\PIF}{\text{PIF}}
\newcommand{\iqr}{\mathcal{IQR}}
\newcommand{\Q}{\mathcal{Q}}
\newcommand{\V}{\mathcal{V}}
\newcommand{\R}{\mathcal{R}}
\newcommand{\T}{\mathcal{T}}
\newcommand{\w}{\mathcal{w}}
\mathcode`w=`\w

\renewcommand\arraystretch{0.65}
\renewenvironment{proof}{{\bfseries\noindent Proof.}}{\qed}

\begin{document}

\title{Interval estimators for ratios of independent quantiles and interquantile ranges}

\author{Chandima N. P. G. Arachchige \\ Department of Mathematics and Statistics, La Trobe University \\ \href{mailto:18201070@students.latrobe.edu.au}{18201070@students.latrobe.edu.au}\\\\ Maxwell Cairns\\ Department of Mathematics and Statistics, La Trobe University \\ \href{mailto:mrcairns@students.latrobe.edu.au}{mrcairns@students.latrobe.edu.au} \\ \\ Luke A. Prendergast\\ Department of Mathematics and Statistics, La Trobe University \\ \href{mailto:luke.prendergast@latrobe.edu.au}{luke.prendergast@latrobe.edu.au}}

\maketitle

\begin{abstract}
Recent research has shown that interval estimators with good coverage properties are achievable for some functions of quantiles, even when sample sizes are not large.   Motivated by this, we consider interval estimators for the ratios of independent quantiles and interquantile ranges that will be useful when comparing location and scale for two samples.  Simulations show that the intervals have excellent coverage properties for a wide range of distributions, including those that are heavily skewed.  Examples are also considered that highlight the usefulness of using these approaches to compare location and scale. 
\end{abstract}


\section{Introduction}

It is common to use the t-test to compare the differences between the means of two independent populations and under the assumption of normality of those populations. However, when the distributions are skewed medians may be a more appropriate measure of location. Non-parametric alternatives to the t-test, such as tests for the difference, or ratio of, two medians are available \citep[e.g.][]{price2002distribution}. Similarly, when comparing the spread of two populations, normality is again often assumed and the standard F-test employed based on the ratio of two independent sample variances.  However, it has long been known that the F-test can be unreliable when normality is violated \citep[see, e.g.][]{brown1974robust}.  For a recent discussion see \cite{hosken2018beware} who advise ``do not use F tests to compare variances''.   As a non-parametric alternative, \cite{shoemaker1999interquartile} introduced a test using differences in \textit{interquantile ranges}.  They found that the test is reliable for many distributions, including those are heavily skewed when location is known or unknown.

In this paper, we propose interval estimators for ratios of quantiles and interquantile ranges which are scale-free and thus easily interpretable.  We begin by detailing some related existing methods in Section 2 which also allows us to introduce notations used throughout.  In Section 3 we obtain partial influence functions for ratios of independent quantile and interquantile range estimators.  Partial influence functions, which can be used to study robustness properties, can also be used to obtain asymptotic variances of estimators.  It is the latter which is our main focus and the asymptotic variances are used to construct interval estimators in Section 4.  We conduct simulations and then provide some examples in Section 5.

\section{Notations and related existing methods}\label{sect:Notations and related existing methods}

Let $F_1$ denote the distribution function for random variable $X$ and $f_1$ denote the density. For a $p \in [0,\ 1]$, let the $p$th quantile be $x_p = F_1^{-1}(p)=\inf \{x:F_1(x) \geq p\}$.  Also, let $g_1(p)=1/f_1\left(x_p\right)$ denote the \textit{quantile density function} \citep{tukey1965,parzen1979nonparametric} and its reciprocal,  which we denote $q_1(p)=f_1\left(x_p\right)$, is the \textit{density quantile function} . Similarly, let $F_2$ denote the distribution function for random variable $Y$ with $y_p = F_2^{-1}(p)=\inf \{y:F_2(y) \geq p\}$, $g_2(p)=1/f_2\left[y_p\right]$ and $q_2(p)=f_2\left(y_p\right)$.  Also let $X_1,\ldots,X_{n_1}$ and $Y_1,\ldots,Y_{n_2}$ denote simple random samples of size $n_1$ and $n_2$ from $F_1$ and $F_2$ respectively.  

\subsection{\textit{The Price and Bonett method}} \label{sect:BP method}

\cite{price2002distribution} proposed an asymptotic confidence interval for a ratio of medians which does not require identically shaped distributions. Let $X_{(1)} \leq X_{(2)} \leq \ldots \leq X_{(n_1)}$ and $Y_{(1)} \leq Y_{(2)} \leq  \ldots \leq Y_{(n_2)}$ be the ordered random samples. Let $\widehat{\eta}_1$ and $\widehat{\eta}_2$ be the usual sample medians obtained from each sample which are estimators of $\eta_1$ and $\eta_2$. Let $X^*_{(i)}=\ln X_{(i)}$ and  $Y^*_{(i)}=\ln Y_{(i)}$, assuming $X_{(i)}$ and $Y_{(i)}$ are both non-negative and let $\widehat{\eta}^*_1$ and $\widehat{\eta}^*_2$ denote the sample medians of these log-transformed samples. An asymptotic distribution-free confidence interval for $\eta_1/\eta_2$ defined as,
\begin{align*}
\left(\frac{\hat{\eta}_1}{\hat{\eta}_2}\right) \exp\left\{\pm Z_{\alpha/2}\sqrt{\text{Var}(\hat{\eta}^*_1)+\text{Var}(\hat{\eta}^*_2)}\right\}
\end{align*}
where $\text{Var}(\hat{\eta}_j^*)$ $(j=1,2)$ is the variance of $\hat{\eta}_j^*$.  The \citep{price2001estimating} modification of the McKean-Schrader estimator \cite{mckean1984comparison} is used where 
\begin{align*}
\text{Var}(\hat{\eta}_1^*) = \left(\frac{(X^*_{(n_1-c_1+1)}- X^*_{(c_1)})}{2z_1}\right)^2
\end{align*}
where $c_1 = (n_1+1)/2 - n_1^{1/2}$, rounded to the nearest integer, $z_1 =\Phi^{-1} (1-p_1/2)$ and $p_1=\sum_{i=0}^{c_1-1}\left[n_1!/i!(n_1-i)!\right](0.5)^{(n_1-i)}$. $\text{Var}(\hat{\eta}_2^*)$ is similarly defined.

\subsection{\textit{Shoemaker's test}} \label{sect:shoemaker}

For $p\in (0, 0.5)$, the interquantile range is denoted $ \IQR_p(X)=x_{1-p} - x_p $ which is the usual \textit{interquartile range} when $p=0.25$.   Let $\widehat{x}_p$ denote the estimator of $x_p$.  \cite{shoemaker1995tests} uses the influence function \citep{hampel1974influence} to calculate the asymptotic variance of the $\IQR_p(X)$ estimators.  We will provide more detail on the influence function in the next section.  Using our notations for the density quantile function at the start of this section, the asymptotic variance for $\IQR_p(X)$ estimator is $\omega_1^2=p\left\{q_1(p)+q_1(1-p)-p[q_1(p)+q_1(1-p)]^2\right\}/[q_1^2(p)q_1^2(1-p)]$.  Hence, for $\omega_2$ denote the asymptotic variance for the estimator of $\IQR_p(Y)$, the \cite{shoemaker1999interquartile} test statistic is
\begin{equation}\label{T}
Z=\frac{(\widehat{x}_{1-p} - \widehat{x}_{p}) - (\widehat{y}_{1-p} - \widehat{y}_{p})}{\sqrt{\omega^2_1/n_1+\omega^2_2/n_2}}.
\end{equation}
Z is asymptotically $N(0,1)$ distributed when $\IQR_p(X)=\IQR_p(Y)$ provided that $f_1$ and $f_2$ are positive and continuous for the quantiles used in the interquantile ranges.  As the estimator of the denominator for $q_1(p)$, and then similarly $q_1(1-p)$, $q_2(p)$ and $q_2(1-p)$, \cite{shoemaker1999interquartile} uses $n(\widehat{x}_p, h_n)/(2nh_n)$ where $n(\widehat{x}_p, h_n)$ is the number of observations falling in the interval $\widehat{x}_p\pm h_n$ with bandwidth $h_n=1.3s/n^{1/5}$.  Simulation results for a wide variety of distributions validate the use of this test, and certainly superiority over the $F$-test in the presence of skew, and it is suggested that one should choose $p$ between 0.1 and 0.25 for improved power.

\section{Ratios of quantiles and interquantile ranges} \label{Ratios of quantiles and interquantile ranges}

In this section we introduce the ratio estimators and ultimately derive their asymptotic variances.

\subsection{\textit{The ratio estimators}} \label{sect:ratio estimators}

We continue with the notations already introduced and assume $p\in (0,1)$. We define the population ratio of quantiles $r_p$ and associated estimator to be
\begin{equation}\label{r}
r_p= \left(\frac{x_p}{y_p}\right)\;\;\text{and}\;\; \widehat{r}_p = \left(\frac{\hat{x}_p}{\hat{y}_p}\right).
\end{equation}

Assuming $p\in (0,0.5)$, we define the population squared ratio of $\IQR_p$s and associated estimator to be
\begin{equation}\label{R}
R_p = \left[\frac{\IQR_p(X)}{\IQR_p(Y)}\right]^2\;\;\text{and}\;\; \widehat{R}_p = \left(\frac{\widehat{x}_{1-p} - \widehat{x}_p}{\widehat{y}_{1-p} - \widehat{y}_p}\right)^2.
\end{equation}

We focus on the squared ratio of IQRs since it is analogous to the ratio of variances although it is simple to obtain estimators for the ratio of IQRs by a square-root transformation. A nice property of $R_p$ is that it is equivalent to the ratio of variances for many distributions, as shown below.

\begin{lemma}\label{lemma}
Let $G(\mu,\sigma)$ be the distribution function of a location-scale family with location and scale parameters $\mu$ and $\sigma$.  If $X\sim G(\mu_1,\sigma_1)$ and $Y\sim G(\mu_2,\sigma_2)$,
$R_p= \Var(X)/\Var(Y)$ for any $p\in (0,1/2)$.
\end{lemma}

\begin{proof}
The proof is obvious when noting that for $Z_1\sim G(0,1)$ and $Z_2\sim G(0,1)$, we can write $X=\sigma_1 Z_1 + \mu_1 $ and $Y=\sigma_2 Z_2 + \mu_2$ so that the quantile functions for $X$ and $Y$ may each be written $\sigma_1 Q(p) + \mu_1$ and $\sigma_2 Q(p) + \mu_2$ where $Q(p)$ is the quantile function for $G(0,1)$.  Hence, $R_p=\sigma_1^2/\sigma_2^2$.  For $Z\sim G(0,1)$, $\Var(X)=\sigma_1^2\Var(Z)$ and $\Var(Y)=\sigma_2^2\Var(Z)$ so that $R_p=\Var(X)/\Var(Y)$. 
\end{proof}

From Lemma \ref{lemma}, $R_p$ is equal to the ratio of variances when the distributions are from the same location-scale family.  This means that an estimator of $R_p$ is a direct competitor to the ratio of variances for such distributions.

\subsection{\textit{Influence functions and partial influence functions}} \label{sect:IF and PIF}

Let $F$ denote a distribution function and for $\epsilon \in [0,1]$, define the contamination distribution to be $F_{\epsilon}=(1-\epsilon)F+\epsilon \Delta_{x_0}$ where $\Delta_{x_0}$ has all of its mass at the contaminant $x_0$. Suppose that for $F$ there is a parameter of interest, $\theta$, and associated estimator with the statistical functional $T$ such that $T(F)=\theta$ and $T(F_n)=\widehat{\theta}$.  For example, for the mean parameter $\mu$, we have $\mu=T(F)=\int xdF$.  The influence function \citep[IF,][]{hampel1974influence} is then defined to be $$\IF (x_0;T,F) \equiv \lim _{\epsilon \to 0}\frac{T(F_\epsilon)-T(F)}{\epsilon} =\left. \frac {\partial}{\partial \epsilon} T( F_\epsilon)\right|_{\epsilon = 0}$$ which is the rate of change in $T$, at $F$, when a small amount of contamination is introduced.  Influence functions are therefore useful tools to understand the behavior of estimators in the presence of certain observations types, including outliers. 

Let $f$ denote the probability density function of $F$ and let $\Q_p$ denote the functional for the $p$th quantile where $\Q_p(F)=x_p$.  The influence function of the $p$th quantile is well known \cite[e.g., p.59 of ][]{staudte1990robust} to be
\begin{equation}
\IF (x_0;\Q_p,F)= \left[p-I(x_p\geq x_0)\right]g(p)\label{IFQ}
\end{equation} where $g(p)= 1/f(x_p)$ is the quantile density defined earlier \citep{parzen1979nonparametric}. 

Influence functions also exhibit useful asymptotic properties including an often convenient means to derive asymptotic variances such as those computed for the IQR by \cite{shoemaker1995tests}.  For $X\sim F$ and $F_n$ denoting the empirical distribution for $n$ iid random variables distributed $F$, under some mild regularity conditions such as differentiability of $T(F)$ and by the Central Limit Theorem we have \citep[see, e.g., page 63 of][]{staudte1990robust},
\begin{equation}
 \sqrt{n}\left[T(F_n)-T(F)\right]\adistr N\left(0, \text{ASV}(T)\right)
\end{equation}
where $\adistr$ denotes `approximately distributed as' and $\text{ASV}(T)=E\left[\IF (X;T,F)^2\right]$ is the asymptotic variance of the estimator with functional $T$.

For the quantile estimator with functional $Q_p$ and influence function given in \eqref{IFQ}, it can be shown that $E_F[\IF (X;\Q_p,F)]=0$ and 
\begin{equation}
\text{ASV}(Q_p)=E_F[\IF^2 (X;\Q_p,F)]=p(1-p)g^2(p).
\end{equation} 

In our context, we have two populations and therefore consider partial influence functions \citep[PIF,][]{pires2002partial}. We have two $\PIF$s, where contamination is introduced to each of the populations while the other population remains uncontaminated. The first $\PIF$ of the estimator functional $T$ at ($F_1,F_2$) is
\begin{align*}
\PIF_1(x_0;T,F_1,F_2) = \lim_{\epsilon \to   0}\left[\frac{T[(1-\epsilon)F_1+\epsilon\Delta_{x_0},F_2]-T(F_1,F_2)}{\epsilon}\right],  
\end{align*}
with $\PIF_2(x_0;T,F_1,F_2)$ defined similarly.  Let $F_{n_1}$ and $F_{n_2}$ denote empirical distribution functions for iid samples of size $n_1$ and $n_2$ from $F_1$ and $F_2$ then from \cite{pires2002partial} we have that
$\sqrt{n_1 + n_2}\left[T(F_{n_1}, F_{n_2})-T(F_1,F_2)\right]$
is asymptotically normal with mean zero and asymptotic variance (ASV)
\begin{equation}
  \ASV(T)=\frac{1}{w_1}E_{F_1}[\PIF_1(X;T,F_1,F_2)^2]+ \frac{1}{w_2}E_{F_2}[\PIF_2(X;T,F_1,F_2)^2]\label{asv:RPIF}
\end{equation}
where $w_i=n_i/(n_1+n_2)$ $(i=1,2)$ and $E_F(.)$ denotes expectation when $X\sim F$.

\subsubsection*{\textit{Partial Influence Functions for the ratio of quantiles}} \label{sect:PIF_rp}

Let $\mathsf{r}_p$ be the functional for ratio of quantiles so that 
$$\mathsf{r}_p(F_1,F_2)=\left[\frac{\Q_{p}(F_1)}{ \Q_{p}(F_2)}\right]=r_p$$  Recall $\Q_p(F_1)=x_p$ and $\Q_p(F_2)=y_p$ to distinguish between the populations.

\begin{theorem}\label{th:IFrp}
For $\IF (x_0;\Q_{p},F)$ is defined in \eqref{IFQ}, the partial influence functions of $\mathsf{r}_p$ for contamination introduced to each of $F_1$ and $F_2$ are
\begin{equation*}
\PIF_1(x_0;\mathsf{r}_p,F_1,F_2) = \left[\frac{\IF (x_0;\Q_p,F_1)}{y_p}\right],\;\; \PIF_2(x_0;\mathsf{r}_p,F_1,F_2) = -r_p\left[\frac{\IF (x_0;Q_p,F_2)}{y_p}\right].
 \end{equation*}
\end{theorem}
The proof of Theorem \ref{th:IFrp} is in Section \ref{app:proof_of_IFrp}.

\subsubsection*{\textit{Partial Influence Functions for the ratio of variances} } \label{sect:PIF_R}

Let the mean and variance of the distribution described by $F$ be $\mu$ and $\sigma^2$.  For $\T$ denoting the functional for the mean estimator, we have $\T(F)=\int xdF=\mu$. Let $\V$ be the functional for the variance estimator where $\V(F)=\int \left[x-T(F)\right]^2 dF=\sigma^2$.  Then $\T(F_\epsilon)=\int xd\left[(1-\epsilon)F+\epsilon \Delta_{x_0}\right]=(1-\epsilon)\mu +\epsilon x_0$ and $\V(F_\epsilon)=\sigma^2+\epsilon(1-\epsilon)\left[(x_0-\mu)^2 - \sigma^2\right].$ Consequently, $\IF(\V,F,x)=(x-\mu)^2-\sigma^2$ is the influence function for the variance estimator. 

Let $\T(F_j)=\mu_j$, $\V(F_j)=\sigma_j^2$ $(j=1,2)$. For $\R$ denoting the functional for the ratio of variances, we have $\R(F_1,F_2)=\V(F_1)/\V(F_2)=\sigma_1^2/\sigma_2^2=\rho$.  Then, for $z_j=(x_0-\mu_j)/\sigma_j$ $(j=1,2)$, the $\PIF$s for $\R$ are
\begin{equation}
\PIF_1(x_0;\R,F_1,F_2) = \rho(z^2_1-1),\;\;\PIF_2(x_0;\R,F_1,F_2) =-\rho(z^2_2-1)\label{IFR}
\end{equation}
As expected, the $\PIF$s are unbounded in $x_0$ indicating that outliers can exert unlimited influence.

\subsubsection*{\textit{Partial Influence Functions for the squared IQR Ratio}} \label{sect:PIF_Rp}

Let $\R_p$ be the functional for the squared ratio of $\IQR$s so that 
$$\R_p(F_1,F_2)=\left[\frac{\Q_{1-p}(F_1) - \Q_{p}(F_1)}{\Q_{1-p}(F_2) - \Q_{p}(F_2)}\right]^2=R_p.$$  

\begin{theorem}\label{th:IFRp}
For $\IF (x_0;\Q_{p},F)$ is defined in \eqref{IFQ}, the partial influence functions of $\R_p$ for contamination introduced to each of $F_1$ and $F_2$ are
\begin{align*}
\PIF_1(x_0;\R_p,F_1,F_2) &=  \frac{2R_p}{x_{1-p} - x_p}\left[\IF (x_0;\Q_{1-p},F_1) - \IF (x_0;\Q_p,F_1)\right],\\
 \PIF_2(x_0;\R_p,F_1,F_2) &= -\frac{2R_p}{y_{1-p} - y_p}\left[\IF (x_0;\Q_{1-p},F_2) - \IF (x_0;Q_p,F_2)\right].
 \end{align*}
\end{theorem}
The proof of Theorem \ref{th:IFRp} is in Section \ref{app:proof_of_IFRp}.

\subsubsection{\textit{Asymptotic variances}} \label{sect:ASV}

Recall that $\mu_i$ and $\sigma_i$ denote the mean and standard deviation of $F_i$ $(i=1,2)$ and that $\rho=\sigma_1^2/\sigma_2^2$.  Then from \eqref{asv:RPIF} and \eqref{IFR}, it is straight forward to show that the ASV for the ratio of variances estimator is
\begin{equation}
\ASV(\R;n_1,n_2)=\rho^2 \left\{\frac{1}{w_1}[E_{F_1}(Z_1^4)-1]+\frac{1}{w_2}[E_{F_2}(Z_2^4)-1]\right\}\label{ratio_of_var}
\end{equation}
where $Z_i=(X-\mu_i)/\sigma_i$ so that $E_{F_i}(Z_i^4)$ is the scaled fourth central moment of $F_i$ $(i=1,2)$.  Recall that $g_1(p)=1/f_1(x_p)$ and $g_2(p)=1/f_1(p)$ are the quantile density functions. We now provide the ASV for the  ratio of quantiles.

\begin{theorem}\label{th:ASVrp}
The asymptotic variances of $\sqrt{n_1 + n_2}\mathsf{r}_p(F_{n_1}, F_{n_2})$ and $\sqrt{n_1 + n_2}\R_p(F_{n_1}, F_{n_2})$ are
\begin{align*}
\ASV(\mathsf{r}_p;n_1,n_2)=& p(1-p)r_p^2\Bigg\{\frac{g_1^2(p)}{w_1 x_p^2}
+\frac{g_2^2(p)}{w_2y_p^2} \Bigg\}.
\end{align*}
and
\begin{align*}
\ASV(\R_p;n_1,n_2)=& 4p\rho_p^2\Bigg\{\frac{g_1^2(p) + g_1^2(1-p) - p \left[g_1(p) + g_1(1-p)\right]^2}{w_1(x_{1-p} - x_p)^2}\\
&+\frac{g_2^2(p) + g_2^2(1-p) - p \left[g_2(p) + g_2(1-p)\right]^2}{w_2(y_{1-p} - y_p)^2} \Bigg\}.
\end{align*}
\end{theorem}
\noindent The proof of Theorem \ref{th:ASVrp} is in Section \ref{app:proof_of_ASVrp}

\begin{corollary}\label{corr:asv}
Suppose that $X$ and $Y$ are both random variables from the same location-scale family such that the density of $X$ may be written $f(x;\mu_1, \sigma_1)$ and the density of $Y$ $f(y;\mu_2, \sigma_2)$ where $\mu_1$, $\mu_2$ and $\sigma_1$, $\sigma_2$ are the respective location and scale parameters.  Let $q_{1-p}$ and $q_p$ denote the $(1-p)$th and $p$th quantiles of the distribution with density $f(\cdot;0,1)$ and $g_0(1-p)=1/f(q_{1-p};0,1)$ and $g_0(p)=1/f(q_{p};0,1)$ the respective quantile densities.  Then $$\ASV(\R_p;n_1,n_2)= 4p\frac{\sigma_1^4}{\sigma_2^4}\Bigg\{\frac{g_0^2(p) + g_0^2(1-p) - p \left[g_0(p) + g_0(1-p)\right]^2}{w_1(1-w_1)(q_{1-p} - q_p)^2}\Bigg\}.$$
\end{corollary}

\begin{proof}
Since $X$ and $Y$ are from the same location-scale family, then $x_{1-p}-x_p=\eta_1 (q_{1-p} - q_p)$, $y_{1-p}-y_p=\eta_2 (q_{1-p} - q_p)$ and $$f(x;\mu_1,\eta_1)=\frac{1}{\sigma_1}f\left(\frac{\sigma_1 x - \mu_1}{\sigma_1};0,1\right), f(y;\mu_2,\sigma_2)=\frac{1}{\sigma_2}f\left(\frac{\sigma_2 y - \mu_2}{\sigma_2};0,1\right).$$  Using these results $g_1(p)=g_0(p)\sigma_1$ and $g_2(p)=g_0(p)\sigma_2$.  The result follows after some simplification and noting that $w_2=1-w_1$.
\end{proof}

\begin{remark}\label{remark}
Since the ASV in Corollary \ref{corr:asv} depends on location and scale only through $\sigma_1^4/\sigma_2^4$ which is a common factor to all terms, then the choice of $p$ that minimizes the ASV is independent of the location and scale parameters.
\end{remark}

 \begin{table}[h]
      \centering
    \caption{Choice of $p$ to minimize ASV for the squared IQR ratio.}\label{table:minp}
    \begin{tabular}{cccccc}
    \hline
    Distribution & $p$ & Distribution & $p$ & Distribution & $p$ \\
    \hline
    Exp$(\lambda)$ & 0.128 &  Beta(0.1,0.1) & 0 & Gamma(1) & 0.128 \\
    Unif$(a,b)$ & 0 & Beta(0.5,0.5) & 0 & Gamma(2) & 0.110 \\
    Log Normal(0,1) & 0.193 & Beta(1,1) & 0 & Gamma(10) & 0.081 \\
    Log Normal(1,1) & 0.193 & Beta(10,10) & 0.055 & PAR(1) & 0.282 \\
    N$(\mu,\sigma^2)$ & 0.069 & Weibull(0.5) & 0.181 & PAR(2) & 0.224 \\
    Chi-Squared(1) & 0.127 & Weibull(1) & 0.128 & PAR(3) & 0.198 \\
    Chi-Squared(2) & 0.128 & Weibull(2) & 0.069 & PAR(5) & 0.173 \\
    Chi-Squared(25) & 0.079 & Weibull(10) & 0.081 & PAR(7) & 0.161\\
    \hline
    \end{tabular}  
  \end{table}
  
We now explore the choices of $p$ that result in the minimum ASV of the squared IQR ratio estimator for several distributions. As shown in Table \ref{table:minp}, the choice of $p$ that minimizes the ASV varies for different distributions.  The choice of $p$ that minimizes the ASV for the exponential, uniform and normal distributions does not depend on the parameters of these distributions (see Remark \ref{remark} which is a consequence of Corollary \ref{corr:asv}).  For distributions considered with the exception of small shape parameter for the Pareto type II (PAR), choosing a $p< 0.25$ gives a smaller ASV than if one were to use the ratio of interquartile ranges.  This agrees with the observations of \cite{shoemaker1999interquartile}.  Our interest is mainly in applications to skewed data and we therefore favor $p=0.2$ which would give good results for log normal and Pareto-type distributions. 
 
\section{Interval Estimators} \label{sect:IntEst}

To estimate the quantiles, we use the default continuous sample quantile from the R \texttt{quantile} function.  To estimate the quantile density, $g_1$ and then similarly for $g_2$, we use the kernel density estimator
$$\widehat{g}_1(p,b)=\sum^n_{i=1}X_{(i)}\left\{k_b\left(p - \frac{i-1}{n_1}\right) - k_b\left(p-\frac{i}{n_1}\right)\right\}$$
with kernel function $k_b$, for which we use the \cite{epan-1969} kernel, and bandwidth $b$.  It has recently been shown that excellent confidence interval coverage for estimators of functions of quantiles can be obtained for sample sizes even as low as 30 \citep[see][]{prendergast2016quantileLorenz, prendergast2017largen, prendergast2017simple}.  These works use the Quantile Optimality Ratio \citep[QOR, ][]{prendergast2016exploiting} to choose the optimal $b$ for estimaing the quantile densities.  We therefore use the QOR in selecting our $b$ although other choices of $b$ are also possible.  

\subsection{\textit{Approximate variances of the estimators}} \label{sect:ApproxVar}

Let $\widehat{\rho}=S_1^2/S_2^2$ be the estimator of $\sigma_1^2/\sigma_2^2$ where $S_i^2 = \V(F_{n_i})$ $(i=1,2)$ are the sample variance estimators.  Let $\{X_i\}^{n_1}_{i=1}$ denote the simple random sample for the first sample with sample mean $\overline{X}=\T(F_{n_1})$ and $\{Y_i\}^{n_2}_{i=1}$ denote the simple random sample for the second with  $\overline{Y}=\T(F_{n_2})$. From \eqref{ratio_of_var},
\begin{align*}
  \Var\left(\widehat{\rho}\right)&\approx\frac{\widehat{\rho}^2}{n_1+n_2} \left[\frac{1}{w_1}\left(\overline{Z_1^4}-1\right)+\frac{1}{w_2}\left(\overline{Z_2^4}-1\right)\right]
\end{align*}
where
$$w_i=\frac{n_i}{n_1+n_2},\ \overline{Z_1^4}= \frac{1}{n_1}\sum^{n_1}_{i=1}\left(\frac{X_i-\bar{X}}{S_1}\right)^4, \ \overline{Z_2^4}= \frac{1}{n_1}\sum^{n_1}_{i=1}\left(\frac{Y_i-\bar{Y}}{S_2}\right)^4,\ \widehat{\rho} = \frac{S_1^2}{S_2^2}.$$

Let $\Var(\widehat{r}_p)$ denote the variance of the ratio of quantiles estimator. Then, from Theorem \ref{th:ASVrp}
\begin{align*}
\Var(\widehat{r}_p)=& \frac{p(1-p)\widehat{r}_p^2}{n_1+n_2}\Bigg\{\frac{\widehat{g}_1^2(p)}{w_1 \widehat{x}_p^2}
+\frac{\widehat{g}_2^2(p)}{w_2\widehat{y}_p^2} \Bigg\},
\end{align*}
where $\widehat{g}_i(p)$ $(i=1,2)$ is the estimated quantile density using the QOR method. Similarly, let $\Var(R_p)$ denote the variance of the squared ratio of IQRs estimator. Then, from Theorem \ref{th:ASVrp}
\begin{align*}
\Var(\widehat{\rho}_p)\approx &\ \frac{4\ p\widehat{R}_p^2}{n_1+n_2}\Bigg\{\frac{\widehat{g}_1^2(p) + \widehat{g}_1^2(1-p) - p \left[\widehat{g}_1(p) + \widehat{g}_1(1-p)\right]^2}{w_1(\widehat{x}_{1-p} - \widehat{x}_p)^2}\nonumber\\
&+\frac{\widehat{g}_2^2(p) + \widehat{g}_2^2(1-p) - p \left[\widehat{g}_2(p) + \widehat{g}_2(1-p)\right]^2}{w_2(\widehat{y}_{1-p} - \widehat{y}_p)^2} \Bigg\}.
 \end{align*}

\subsection{\textit{Asymptotic confidence intervals}} \label{sect:CIs}

In constructing our interval estimators for the ratios we use the log transformation and exponentiate to return to the ratio scale.  For a random variable $W>0$, using the Delta method it follows that $\Var[\ln(W)]\approx \Var(W)/W^2$.  Hence, approximate $(1-\alpha) 100$\% confidence intervals for the ratio of quantiles, ratio of variances and squared ratio of IQRs are
\begin{equation}
\exp\Bigg[\ln(\widehat{r}_p)\pm  z_{1-\alpha/2}\frac{1}{\widehat{r}_p}\sqrt{\Var\left(\widehat{r}_p\right)}\Bigg]\label{CIs1} 
\end{equation}
\begin{equation}
\exp\Bigg[\ln(\widehat{\rho})\pm  z_{1-\alpha/2}\frac{1}{\widehat{\rho}}\sqrt{\Var\left(\widehat{\rho}\right)}\Bigg] \ and \
\exp\Bigg[\ln(\widehat{\rho}_p)\pm  z_{1-\alpha/2}\frac{1}{\widehat{\rho}_p}\sqrt{\Var\left(\widehat{\rho}_p\right)}\Bigg]\label{CIs2}
\end{equation}
where $z_{1-\alpha/2}$ is the $(1-\alpha/2)\times $100 percentile of the standard normal distribution and where for the variances we use the approximations from Section \ref{sect:ApproxVar}.

\section{Simulations and Examples} \label{sect:SimEx} 
\subsection{\textit{Simulations}} \label{sect:Sim} 
A simulation study was conducted to compare the performance among estimators  by considering coverage probability (cp) and the average confidence interval width (w)  as the performance measures. We have selected the lognormal, exponential, chi-square and Pareto distributions with different parameter choices with the sample sizes $n=\{50, 100, 200, 500, 1000\}$  and 10,000 simulation trials to our simulation study.

\renewcommand{\arraystretch}{0.95}
\begin{table}[htbp]
  \centering
  \caption{Simulated coverage probabilities (and widths) for the 95\% confidence interval estimators for the interval based on the Price and Bonnet (rows labeled PB) method and the interval in \eqref{CIs1} ($r_p$) with $p=0.25, 0.5, 0.75$.} 
    \begin{tabular}{crcccc}
    \toprule
    \multicolumn{1}{c}{Sample size} & \multicolumn{1}{l}{$p$} & \multicolumn{1}{l}{$X\sim$ LN(0,1)} & $X\sim $ EXP(1)  & $X\sim \chi^2_5$ & $X\sim$ PAR(1,7)  \\
    \multicolumn{1}{c}{($n_1$,$n_2$)} &       & $Y\sim $ LN(0,1) & $Y\sim$ EXP(1) & $Y\sim \chi^2_2$    & $Y\sim$ PAR(1,3)  \\
    \midrule
    50,50 & PB    & 0.961(1.12) & 0.965(1.31) & 0.964(3.29) & 0.963(0.57) \\
          & $r_{0.25}$ & 0.963(1.18) & 0.959(1.71) & 0.957(6.05) & 0.962(0.75) \\
          & $r_{0.5}$  & 0.972(1.17) & 0.967(1.30) & 0.966(3.20) & 0.973(0.60) \\
          & $r_{0.75}$ & 0.977(1.38) & 0.970(1.16) & 0.965(2.24) & 0.971(0.55) \\
    \midrule
    100,100 & PB     & 0.960(0.75) & 0.962(0.88) & 0.962(2.21) & 0.958(0.38) \\
          & $r_{0.25}$ & 0.965(0.82) & 0.961(1.19) & 0.958(4.31) & 0.959(0.51) \\
          & $r_{0.5}$  & 0.970(0.77) & 0.960(0.87) & 0.958(2.19) & 0.966(0.39) \\
          & $r_{0.75}$ & 0.972(0.89) & 0.962(0.77) & 0.965(1.52) & 0.969(0.35) \\
    \midrule
    200,200 & PB     & 0.952(0.51) & 0.953(0.59) & 0.950(1.50) & 0.954(0.26) \\
          & $r_{0.25}$ & 0.967(0.58) & 0.957(0.82) & 0.957(3.04) & 0.962(0.36) \\
          & $r_{0.5}$  & 0.962(0.53) & 0.961(0.60) & 0.957(1.52) & 0.960(0.26) \\
          & $r_{0.75}$ & 0.967(0.59) & 0.960(0.53) & 0.959(1.04) & 0.966(0.24) \\
    \midrule
    200,500 & PB     & 0.948(0.42) & 0.951(0.49) & 0.953(1.11) & 0.947(0.21) \\
          & $r_{0.25}$ & 0.965(0.48) & 0.953(0.68) & 0.959(2.24) & 0.958(0.30) \\
          & $r_{0.5}$  & 0.961(0.44) & 0.958(0.50) & 0.960(1.12) & 0.961(0.22) \\
          & $r_{0.75}$ & 0.965(0.49) & 0.958(0.43) & 0.961(0.78) & 0.961(0.19) \\
    \midrule
    500,500 & PB     & 0.952(0.32) & 0.947(0.36) & 0.950(0.93) & 0.949(0.16) \\
          & $r_{0.25}$ & 0.962(0.35) & 0.957(0.51) & 0.957(1.90) & 0.955(0.22) \\
          & $r_{0.5}$  & 0.958(0.33) & 0.959(0.37) & 0.957(0.95) & 0.957(0.16) \\
          & $r_{0.75}$ & 0.963(0.36) & 0.954(0.32) & 0.955(0.64) & 0.960(0.14) \\
    \midrule
    500,1000 & PB     & 0.946(0.27) & 0.947(0.31) & 0.946(0.73) & 0.948(0.13) \\
          & $r_{0.25}$ & 0.960(0.31) & 0.953(0.44) & 0.957(1.49) & 0.953(0.19) \\
          & $r_{0.5}$  & 0.960(0.28) & 0.953(0.32) & 0.955(0.74) & 0.959(0.14) \\
          & $r_{0.75}$ & 0.961(0.31) & 0.959(0.28) & 0.954(0.51) & 0.960(0.12) \\
    \midrule
    1000,1000 & PB     & 0.945(0.22) & 0.948(0.25) & 0.945(0.65) & 0.946(0.11) \\
          & $r_{0.25}$ & 0.959(0.25) & 0.958(0.36) & 0.952(1.33) & 0.954(0.15) \\
          & $r_{0.5}$  & 0.955(0.23) & 0.952(0.26) & 0.957(0.66) & 0.954(0.11) \\
          & $r_{0.75}$ & 0.958(0.25) & 0.957(0.23) & 0.954(0.45) & 0.957(0.10) \\
    \bottomrule
    \end{tabular}%
  \label{tab:CP_rp}%
\end{table}%
\renewcommand{\arraystretch}{1}

Simulated coverages based on 10,000 trials for the Price and Bonett method and  interval estimator in \eqref{CIs1} are provided in Table \ref{tab:CP_rp} for several distributions. The Price and Bonett (PB) method for the ratio medians provides very good coverage compared to the nominal 0.95 and the interval width decreases with increasing sample sizes. Similar results can be seen for the ratio of quantiles interval estimator when we choose $p=0.5$ for the ratio of medians.  Coverages suggest that the use of $r_{0.5}$ provides slightly more conservative coverage but with similar interval width.  For ratios of the quartiles ($p=0.25$ and $p=0.75$) coverages are again very good with none reported below the nominal 0.95 and most less than 0.97.  The highest coverages were reported for the smaller sample size setting where $n_1=n_2=50$.

\renewcommand{\arraystretch}{0.95}
\begin{table}[p]
  \centering
  \small
  \caption{Simulated coverage probabilities (and widths) for the 95\% confidence interval estimators for the interval based on the $F$-test (rows labeled $F$) and the intervals in \eqref{CIs2} for the ratio of variances ($R$) and squared ratio of IQRs ($R_p$) with several choices of $p$.
  (*median widths reported due to excessively large average widths after back exponentiation)}
    \begin{tabular}{rrcccc}
    \hline
     \multicolumn{1}{c}{Sample size} & \multicolumn{1}{l}{} & \multicolumn{1}{l}{$X\sim$ LN(0,1)} & $X\sim $ EXP(1)  & $X\sim \chi^2_5$ & $X\sim$ PAR(1,7)  \\
    \multicolumn{1}{c}{($n_1$,$n_2$)} &       & $Y\sim $ LN(0,1) & $Y\sim$ EXP(1) & $Y\sim \chi^2_2$    & $Y\sim$ PAR(1,3)  \\
    \hline
    \multicolumn{1}{c}{50,50} & $F$     & 0.445(2.01) & 0.705(1.39) & 0.756(3.47) & 0.405(0.14) \\
          & $R$      & 0.778(6.11) & 0.867(2.27) & 0.869(4.91) & 0.714(0.35) \\
          & $R_{0.05}$ & 0.975(15.25*) & 0.971(7.67) & 0.969(18.14) & 0.977(0.74*) \\
          & $R_{0.1}$  & 0.978(20.64) & 0.967(4.60) & 0.968(11.44) & 0.977(2.75) \\
          & $R_{0.2}$  & 0.978(7.97) & 0.971(4.24) & 0.971(11.91) & 0.974(0.93) \\
    \hline
    \multicolumn{1}{c}{100,100} & $F$     & 0.389(1.15) & 0.689(0.88) & 0.741(2.19) & 0.341(0.08) \\
          & $R$     & 0.829(4.14) & 0.896(1.59) & 0.903(3.38) & 0.740(0.23) \\
          & $R_{0.05}$ & 0.977(9.57) & 0.970(2.79) & 0.970(6.75) & 0.976(1.03) \\
          & $R_{0.1}$  & 0.975(4.14) & 0.970(2.18) & 0.965(5.81) & 0.975(0.42) \\
          & $R_{0.2}$  & 0.975(3.14) & 0.962(2.17) & 0.967(6.21) & 0.970(0.38) \\
   \hline
    \multicolumn{1}{c}{200,200} & $F$     & 0.348(0.70) & 0.686(0.58) & 0.746(1.47) & 0.296(0.04) \\
          & $R$      & 0.861(2.96) & 0.915(1.10) & 0.926(2.40) & 0.766(0.15) \\
          & $R_{0.05}$ & 0.978(2.99) & 0.968(1.54) & 0.965(3.83) & 0.975(0.27) \\
          & $R_{0.1}$  & 0.973(2.07) & 0.966(1.30) & 0.965(3.50) & 0.968(0.22) \\
          & $R_{0.2}$  & 0.973(1.74) & 0.961(1.31) & 0.963(3.78) & 0.965(0.21) \\
   \hline
    \multicolumn{1}{c}{200,500} & $F$     & 0.340(0.54) & 0.678(0.48) & 0.770(1.20) & 0.314(0.03) \\
          & $R$      & 0.872(2.44) & 0.930(0.91) & 0.937(1.89) & 0.831(0.12) \\
          & $R_{0.05}$ & 0.971(2.12) & 0.965(1.18) & 0.962(2.74) & 0.972(0.17) \\
          & $R_{0.1}$  & 0.969(1.56) & 0.963(1.02) & 0.962(2.60) & 0.968(0.15) \\
          & $R_{0.2}$  & 0.963(1.35) & 0.96(1.03) & 0.961(2.88) & 0.963(0.16) \\
    \hline
    \multicolumn{1}{c}{500,500} & $F$     & 0.324(0.39) & 0.679(0.36) & 0.741(0.90) & 0.257(0.02) \\
          & $R$      & 0.897(1.87) & 0.935(0.70) & 0.939(1.53) & 0.775(0.10) \\
          & $R_{0.05}$ & 0.972(1.40) & 0.963(0.84) & 0.962(2.13) & 0.967(0.13) \\
          & $R_{0.1}$  & 0.965(1.08) & 0.962(0.74) & 0.960(2.01) & 0.964(0.11) \\
          & $R_{0.2}$ & 0.964(0.96) & 0.959(0.76) & 0.960(2.20) & 0.960(0.12) \\
    \hline
    \multicolumn{1}{c}{500,1000} & $F$     & 0.307(0.33) & 0.681(0.31) & 0.765(0.77) & 0.272(0.02) \\
          & $R$      & 0.906(1.63) & 0.940(0.60) & 0.947(1.26) & 0.830(0.08) \\
          & $R_{0.05}$ & 0.966(1.14) & 0.962(0.71) & 0.962(1.70) & 0.970(0.10) \\
          & $R_{0.1}$  & 0.963(0.90) & 0.956(0.63) & 0.962(1.64) & 0.964(0.09) \\
          & $R_{0.2}$  & 0.962(0.81) & 0.961(0.64) & 0.954(1.81) & 0.958(0.10) \\
   \hline
    \multicolumn{1}{c}{1000,1000} & $F$     & 0.302(0.26) & 0.675(0.25) & 0.738(0.63) & 0.215(0.02) \\
          & $R$      & 0.918(1.34) & 0.940(0.49) & 0.941(1.08) & 0.793(0.07) \\
          & $R_{0.05}$ & 0.967(0.89) & 0.963(0.57) & 0.961(1.43) & 0.968(0.08) \\
          & $R_{0.1}$  & 0.963(0.71) & 0.959(0.51) & 0.958(1.38) & 0.960(0.08) \\
          & $R_{0.2}$  & 0.959(0.65) & 0.955(0.52) & 0.956(1.51) & 0.961(0.08) \\
    \hline
    \end{tabular}%
  \label{tab:CP_IQR}%
\end{table}%
\setlength{\parskip}{1.5em}
\renewcommand{\arraystretch}{1}
\parskip 0pt  
Simulated coverages based on 10,000 trials for the F-test and the interval estimators in \eqref{CIs2} are provided in Table \ref{tab:CP_IQR} for several distributions.  The $F$-test approach refers to the standard method for getting an interval for the ratio of variances under the assumption that the data has been sampled from normal distributions.  Consequently, the coverages for the intervals based on the $F$-test (rows labeled $F$) are poor due to the violation of underlying normality.  The interval for ratio of variances using the asymptotic interval provides reasonable coverage for some of the distributions but not when the sample sizes are small to moderate where the intervals appear to be too narrow.  On the other hand, the coverages for the squared IQR ratio interval is very good for all distributions, including for the smaller sample sizes.  For the distributions we have considered here, the squared ratio of IQRs is preferred due to superior coverage.  We have seen this across a broad range of distributions and this can be verified by the reader by using our web application detailed next.  

\subsubsection{\textit{A Shiny web application for the performance comparisons of the intervals}} \label{sect:Shiny}

For further comparisons, we have developed a Shiny \citep{shiny} web application that readers can use to run the simulations with different parameter choices.  This can be found at \url{https://lukeprendergast.shinyapps.io/IQR_ratio/}.

The user can change the distribution, parameters, sample size, probability and the number of trials according to their choices.  Once the desired options are selected the `Run Simulation' button can be pressed and the relevant estimates, coverage probability(cp) and the average width of the confidence interval(w)  will be calculated according to their input choices. 

\subsection{\textit{Examples}} \label{sect:Ex} 
As examples, we have selected three different data sets in different contexts.

\subsubsection{Prostate data} \label{sect:Prost} 

The prostate data set, which we obtained from the \texttt{depthTools} package \citep{dt} in R, is a normalized subset of gene expression data of the 100 most variable genes for 25 randomly selected tumor and 25 randomly selected normal prostate samples from \cite{singh2002gene}. Since the data is not strictly positive meaning ratios of location are not suitable, we restrict our attention to looking for differences in spread between the tumor and normal samples.  Since the sample sizes are comparatively small, we have chosen $p={0.1, 0.2 , 0.25}$ to construct the confidence interval for the ratio of IQRs. We found that there are 6 genes which lead to very different conclusions depending on whether the ratio of variances or ratio of IQRs is used.  These genes, including their abbreviations where applicable, are Carboxylesterase 1 (C1), Glucose-6-phosphate dehydrogenase (G6pd), HDKFZp564A072, S100 calcium-binding protein A4 (S100cbpA4), Selenium binding protein 1 (Sbp1) and Thymosin beta, identified in neuroblastoma cells (Tbiinc).

\begin{figure}[!htb]
 \centering
  \includegraphics[scale = 0.6]{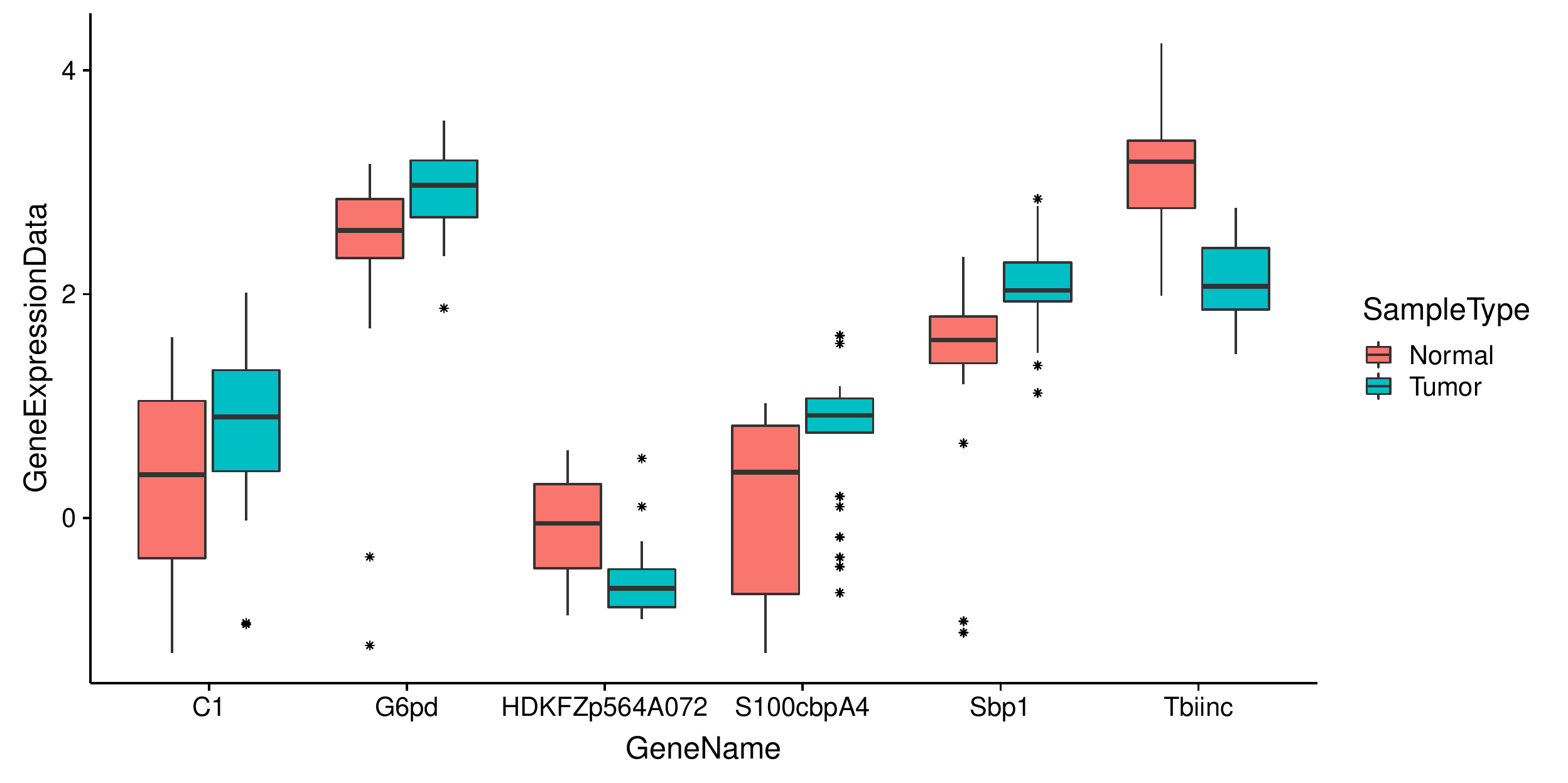} 
  \caption{Box plots of the gene expression data for tumoral and normal samples of selected six genes.}\label{fig:Boxplot_prost}
\end{figure}

Box plots of the genes separated according to groups are shown in Figure \ref{fig:Boxplot_prost}. There is at least one outlier or extreme value in at least one of the two groups in all genes except for Tbiinc.  Ignoring outliers, the boxplots suggest differences in spread for C1, HDKFZp564A072 and S100cbpA4.

\renewcommand{\arraystretch}{0.95}
\begin{table}[htbp]
  \centering
  \caption{95\% asymptotic confidence intervals for the ratio of variances (column labeled $R$) and squared ratios of IQRs ($R_p$) with $p=0.1,0.2,0.25$ for the selected six genes.} 
  \vspace{0.15cm}
    \begin{tabular}{lcccc}
    \toprule
    \multicolumn{1}{l}{Gene} & \multicolumn{1}{c}{ $R$} & \multicolumn{1}{c}{$R_{0.1}$} & \multicolumn{1}{c}{$R_{0.2}$} & \multicolumn{1}{c}{$R_{0.25}$} \\
    \midrule
    C1    & (0.777, 3.016) & (1.206, 5.384) & (2.124, 9.618) & (1.080, 5.426) \\
    G6pd  & (2.085, 20.244) & (0.363, 14.163) & (0.295, 9.564) & (0.140, 8.179) \\
    HDKFZp564A072 & (0.847, 4.397) & (1.533, 8.997) & (1.506, 5.448) & (2.357, 10.065) \\
    S100cbpA4 & (0.950, 3.217) & (1.260, 3.136) & (1.431, 7.786) & (3.211, 183.277) \\
    Sbp1  & (1.193, 10.029) & (0.190, 5.728) & (0.084, 5.387) & (0.172, 11.929) \\
    Tbiinc & (1.163, 4.265) & (0.868, 3.678) & (0.531, 2.611) & (0.533, 2.634) \\
    \bottomrule
    \end{tabular}%
  \label{tab:prost_CI}%
\end{table}%
\renewcommand{\arraystretch}{1.0}

In Table \ref{tab:prost_CI} we provide the estimated asymptotic 95\% confidence intervals for the ratio of variances and squared ratios of IQRs from \eqref{CIs1} and \eqref{CIs2} for the six genes. While the boxplots indicate a difference in spread for  C1, HDKFZp564A072 and S100cbpA4, this is not reflected in the ratio of variances intervals most likely due to the presence of outliers.  However, the differences are captured by the interval estimators for the squared ratio of IQRs.  For the G6pd, Sbpl and Tibiinc genes, the conclusions are reversed where the ratio of variances suggest significant differences in spread while this is not the case for those based on the IQRs.

\subsubsection{Melbourne house price data} \label{sect:HousePrice} 

Since house prices are usually highly skewed, the sample mean is often not indicative of a typical house price.  Therefore, the median is the most popular measure used understanding house price markets. Similarly, the standard deviation is difficult to interpret for skewed data, and IQRs can be more informative when seeking to understand house price spread.  Motivated by this, we now apply our intervals for Melbourne house clearance data from January 2016 obtained from the website \url{https://www.kaggle.com/anthonypino/melbourne-housing-market}.  When describing house prices, it is also common to focus on quartiles \citep[see, e.g.,][]{taylor2011long} so in what follows we choose $p=0.25$ and $p=0.75$.

This data set contains prices of three types of houses (house, unit, townhouse) within different suburbs in Melbourne, Australia. Restricting to suburbs with more than 10 houses sold left data for a total of 301 suburbs.  Our focus will be on comparing house prices between suburbs.  To get an understanding of how often different findings could result depending on whether variances or IQRs were used, we obtained the intervals for every pairing of suburbs result in 45,150 confidence intervals for each ratio.

\begin{table}[htbp]
  \centering
  \caption{Proportion of comparisons giving different conclusions based on ratio of variances/ratio of IQRs ($p=0.25$).  Here we count the number of times that the intervals differ in terms of whether they include one.}
    \begin{tabular}{lll}
    \toprule
    \multicolumn{1}{l}{Type of Conclusion} & \multicolumn{1}{l}{Count} & \multicolumn{1}{l}{Percentage(\%)} \\
    \midrule
    $R$ does not include one and $R_{0.25}$ includes one & 7914  & 17.53 \\
    $R$ includes one and $R_{0.25}$ does not include one  & 5708  & 12.64 \\
    Both intervals include one or do not include one & 31528 & 69.83 \\
    \midrule
    Total & 45150 & 100.00 \\
    \bottomrule
    \end{tabular}%
  \label{tab:HP_prop}%
\end{table}%

The above Table \ref{tab:HP_prop} represents the proportions of times that a different conclusion is reached (assuming conclusions are reached based on whether the ratio intervals include one or not).  We were surprised that just over 30\% of the time there was a difference depending on whether the variances of ratios were used.  This helps to highlight why it one should choose which ratio is best suited their purpose carefully.

\begin{figure}[!htb]
 \centering
 \includegraphics[scale = 0.60]{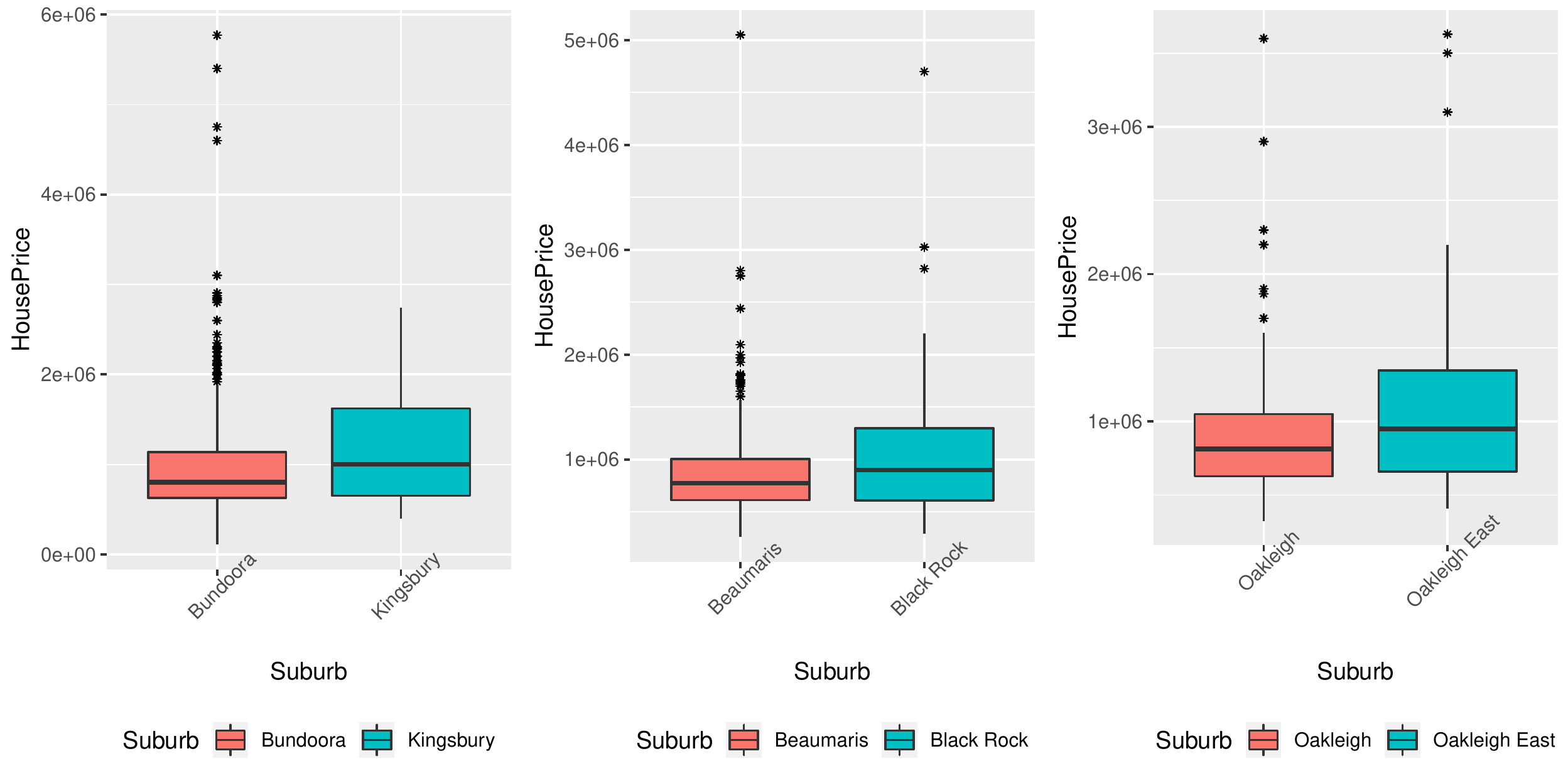}
  \caption{House price comparisons of selected three pair of neighboring suburbs.}\label{fig:HP_neigboxplots}
\end{figure}

To illustrate, Figure \ref{fig:HP_neigboxplots} depicts the house price distribution of selected three pairs of neighboring suburbs.  As expected, the house price distributions are positively skewed and it can be seen there are at least a few outliers in all suburbs except Kingsbury. When considering the middle 50\% of house prices, there are noticeable differences in variation between each of the neighboring suburbs in all three pairings.

\begin{table}[htbp]
  \centering
  \caption{95\% confidence intervals (CI) for the Price and Bonnet method (row labeled PB), asymptotic interval for the ratio of medians ($r$), ratio of variances $R$ and squared ratios of IQRs $R_{0.25}$ for selected three pairs of neighboring suburbs.}  
    \begin{tabular}{lcccccc}
    \toprule
    \multicolumn{1}{c}{Type  } & \multicolumn{2}{c}{Bundoora/Kingsbury} & \multicolumn{2}{c}{Beaumaris/Black Rock} & \multicolumn{2}{c}{Oakleigh/Oakleigh East} \\
    \multicolumn{1}{c}{of} & Estimate & \multicolumn{1}{c}{CI} & Estimate & \multicolumn{1}{c}{CI} & Estimate & \multicolumn{1}{c}{CI} \\
    \midrule
    PB     & 0.801 & (0.628, 1.020) & 0.858 & (0.741, 0.994) & 0.855 & (0.707, 1.035) \\
    $r_{0.25}$ & 0.965 & (0.797, 1.168) & 1.003 & (0.858, 1.172) & 0.952 & (0.809, 1.119) \\
    $r_{0.5}$ & 0.801 & (0.597, 1.074) & 0.858 & (0.745, 0.989) & 0.855 & (0.713, 1.026) \\
     $r_{0.75}$ & 0.670 & (0.537, 0.911) & 0.773 & (0.655, 0.913) & 0.779 & (0.625, 0.971) \\
   $R$     & 1.194 & (0.555, 1.803) & 0.650 & (0.424, 2.361) & 0.698 & (0.436, 2.295) \\
    $R_{0.25}$ & 0.272 & (0.115, 0.643) & 0.325 & (0.166, 0.640) & 0.377 & (0.147, 0.967) \\
    \bottomrule
    \end{tabular}%
  \label{tab:HP_CI}%
\end{table}%

From Table \ref{tab:HP_CI}, the Price and Bonnet and asymptotic methods for the ratio of medians are consistent in their findings.  A significant difference in median house price is detected between Beaumaris and Black Rock, and the intervals for the other pairings, while including the ratio one, also suggest that there may be differences.  The ratio of variance intervals are very wide making it difficult to determine whether there are real differences in spread, despite two of the estimated ratios being substantially less than one.  However, the interval estimators for the squared ratios of interquartle ranges to detect differences in spread and this agrees with other premise that there were notable differences in the spread for the middle 50\% of house prices.  Putting this together and thinking about what it means for a potential home buyer, as an example we consider the Beaumaris and Black Rock neighboring suburbs.  
A typical (median) house in Beaumaris was significantly cheaper (ratio of medians = $0.858$, 95\% CI [0.745, 0.989]) and the spread of prices notably smaller for the middle 50\% of houses.  This reduced spread is also reflected in the approximate equivalent price at the 25th percentile ($\widehat{r}_{0.25}=1.003$, 95\% CI [0.86, 1.17]).

\section{Summary and Discussion} \label{sect:Summary}
We have introduced interval estimators for ratios of quantiles and interquantile ranges.  The intervals have very good coverage, even for samples as small as 50 for a wide range of distributions. Our examples highlight that very different conclusions can be arrived at when using ratios of interquantile ranges instead of ratios of variances.   Future work will consider how to best choose $p$ or the creation of a combined interval that does not require $p$ to be chosen as was done recently by \cite{marozzi2012combined} for hypothesis tests of variation using IQRs.

\appendix

\section{Appendix} \label{sect:Appendix}

\subsection{Proof of Theorem \ref{th:IFrp}}\label{app:proof_of_IFrp}
A power series expansion for $\Q_p(F_\epsilon)$ can be written as $\Q_p(F) + \epsilon \IF(x_0;\Q_{p},F)+O(\epsilon^2)$.  Setting $F_\epsilon=(1-\epsilon)F_1+\epsilon\Delta_{x_0}$ where $\Q_p(F_1)=x_p$ and  for simplicity, write $\IF_{1,p}=\IF(x_0;\Q_{p},F_1)$ and recall $\Q_p(F_2)=y_p$. Then the first PIF is
\begin{small}
\begin{align*}
\PIF_1(x_0;\mathsf{r}_p,F_1,F_2)=&\lim_{\epsilon \downarrow 0}\Bigg\{\frac{x_p + \epsilon\IF_{1,p} + O(\epsilon^2) -x_p}{\epsilon y_p}\Bigg\}=\frac{\IF_{1,p}}{y_p}.
\end{align*}
\end{small}
Setting $F_\epsilon=(1-\epsilon)F_2+\epsilon\Delta_{x_0}$ and letting $\IF_{2,p}=\IF(x_0;\Q_{p},F_2)$ the second PIF is 
\begin{align*}
 \PIF_2(x_0;\mathsf{r}_p,F_1,F_2)=&\lim_{\epsilon \downarrow 0}\left[\frac{x_p/(y_p + \epsilon\IF_{2,p} + O(\epsilon^2))- x_p/y_p}{\epsilon}\right]\\
=&\lim_{\epsilon \downarrow 0}\Bigg\{\frac{x_p[ y_p - (y_p + \epsilon\IF_{2,p} + O(\epsilon^2))]}{\epsilon y_p(y_p + \epsilon\IF_{2,p} + O(\epsilon^2))}\Bigg\}\\
=&\lim_{\epsilon \downarrow 0}\Bigg\{\frac{-x_p[\epsilon\IF_{2,p} + O(\epsilon^2))]}{\epsilon y_p\left[y_p + \epsilon\IF_{2,p} + O(\epsilon^2)\right]}\Bigg\}.
\end{align*}
The proof concludes after canceling the $\epsilon$ terms and taking the limit.

\subsection{Proof of Theorem \ref{th:IFRp}}\label{app:proof_of_IFRp}
A power series expansion for $\Q_{1-p}(F_\epsilon)-\Q_p(F_\epsilon)$ can be written as $$\Q_{1-p}(F)-\Q_p(F) + \epsilon\left[\IF(x_0;\Q_{1-p},F) - \IF(x_0;\Q_{p},F)\right]+O(\epsilon^2).$$  Setting $F_\epsilon=(1-\epsilon)F_1+\epsilon\Delta_{x_0}$ where $\Q_p(F_1)=x_p$, we have $\left[\Q_{1-p}(F_\epsilon)-\Q_p(F_\epsilon)\right]^2$ can be written
\begin{equation}(x_{1-p}-x_p)^2 + 2\epsilon(x_{1-p}-x_p)\left[\IF(x_0;\Q_{1-p},F_1) - \IF(x_0;\Q_{p},F_1)\right] + O(\epsilon^2).\label{IQR2}\end{equation}  For simplicity, write $\PIF_1=\PIF_1(x_0; \R_p,F_1,F_2)$ and $\IF_{1,p}=\IF(x_0;\Q_{p},F_1)$ and recall $\Q_p(F_2)=y_p$. Since $\rho_p=(x_{1-p} - x_p)^2/(y_{1-p} - y_p)^2$, the first PIF is
\begin{small}
\begin{align*}
\PIF_1=&\lim_{\epsilon \downarrow 0}\Bigg\{\frac{(x_{1-p}-x_p)^2 + 2\epsilon(x_{1-p}-x_p)\left[\IF_{1,1-p} - \IF_{1,p}\right] + O(\epsilon^2) -(x_{1-p}-x_p)^2}{\epsilon (y_p-y_{1-p})^2}\Bigg\}\\
=&\frac{2\rho_p}{(x_{1-p} - x_p)}\left[\IF_{1,1-p} - \IF_{1,p}\right].
\end{align*}
\end{small}
Let $\iqr_p(F)=\mathcal{Q}_{1-p}(F)-\mathcal{Q}_p(F)$ be the functional for the IQR at $p$ and for the second PIF set $F_\epsilon=(1-\epsilon)F_2+\epsilon\Delta_{x_0}$.  Then 
\begin{align*}
\PIF_2=&\lim_{\epsilon \downarrow 0}\Bigg\{\frac{(x_{1-p}-x_p)^2\left[\iqr^2(F_\epsilon)\right]^{-1} -(x_{1-p}-x_p)^2/(y_{1-p} - y_p)^2}{\epsilon}\Bigg\}\\
=&\lim_{\epsilon \downarrow 0}\Bigg\{\frac{(x_{1-p}-x_p)^2(y_{1-p} - y_p)^2 -(x_{1-p}-x_p)^2\iqr^2(F_\epsilon)}{\epsilon(y_{1-p} - y_p)^2\iqr^2(F_\epsilon)}\Bigg\}\\
=&\lim_{\epsilon \downarrow 0}\Bigg\{\frac{-2\epsilon(x_{1-p}-x_p)^2(y_{1-p} - y_p)\left[\IF_{2,1-p} - \IF_{2,p}\right] + O(\epsilon^2)}{\epsilon(y_{1-p} - y_p)^2\iqr^2(F_\epsilon)}\Bigg\}
\end{align*}
when using \eqref{IQR2} but evaluated at $F_2$ and letting $\IF_{2,p}=\IF(x_0;\Q_{p},F_2)$. The proof concludes after canceling the $\epsilon$ terms and taking the limit.

\subsection{Proof of Theorem \ref{th:ASVrp}}\label{app:proof_of_ASVrp}

Note $\IF(x_0;\Q_p, F_1)^2=\left[p^2+(1-2p)I(x_p\geq x_0)\right]g_1^2(p)$.  Then, from above and Theorem \ref{th:IFrp} and noting that, for example, $E_{F_1}[I(x_p\geq X)]=p$, for $X\sim F_1$,
\begin{align*}
E_{F_1}\left[\PIF_1(x_0;\mathsf{r}_p,F_1,F_2)^2\right] = & \frac{\mathsf{r}_p^2}{x_p^2}E\left[\IF(X;\Q_p,F_1)^2\right]=\frac{p(1-p)r_p^2}{x_p^2}\left\{g_1^2(p)\right\}.
\end{align*}
$E_{F_2}\left[\PIF_2(x_0;\mathsf{r}_p,F_1,F_2)^2\right]$ is derived similarly and the asymptotic variance follows by applying \eqref{asv:RPIF}.

For the ratio of IQRs, first note $\IF(x_0;\Q_{1-p}, F_1)^2=\left[(1-p)^2-(1-2p)I(x_{1-p}\geq x_0)\right]g_1^2(1-p)$. $\IF(x_0;\Q_p, F_1)\IF(x_0;\Q_{1-p}, F_1)=p\left[(1-p)-I(x_{1-p}\geq x_0)+I(x_p\geq x_0)\right]g_1(p)g_1(1-p)$ since $p < 1-p$. For simplicity let $\IF_p(X)=\IF(X;\Q_p,F_1)$.  Then, from above and Theorem \ref{th:IFRp} and noting that, for example, $E_{F_1}[I(x_p\geq X)]=p$, for $X\sim F_1$,
\begin{align*}
E_{F_1}\left[\PIF_1^2\right] = & \frac{4\rho_p^2}{(x_{1-p} - x_p)^2}\Big\{E\left[\IF^2_{1-p}(X)\right] + E\left[\IF^2_{p}(X)\right]- 2E\left[\IF_{1-p}(X)\IF_{p}(X)\right]\Big\}\\
=&\frac{4p\rho_p^2}{(x_{1-p} - x_p)^2}\left\{g_1^2(p) + g_1^2(1-p) - p \left[g_1(p) + g_1(1-p)\right]^2\right\}.
\end{align*}
$E_{F_2}\left[\PIF_2^2\right]$ is derived similarly and the asymptotic variance follows by applying \eqref{asv:RPIF}.

\bibliographystyle{authordate4}
\bibliography{ref}

\end{document}